\documentclass[11pt]{amsart}

\usepackage[letterpaper,margin=1.08in]{geometry}
\usepackage{amsmath,amssymb,mathtools}
\usepackage{microtype}
\usepackage[colorlinks=true,linkcolor=blue,citecolor=blue,urlcolor=blue]{hyperref}

\newcommand{\R}{\mathbb R}
\newcommand{\Sph}{\mathbb S}
\newcommand{\Per}{\operatorname{Per}}

\newcommand{\id}{\mathrm{Id}}
\newcommand{\dd}{\,d}

\theoremstyle{plain}
\newtheorem{theorem}{Theorem}[section]
\newtheorem{proposition}[theorem]{Proposition}
\newtheorem{lemma}[theorem]{Lemma}
\newtheorem{corollary}[theorem]{Corollary}
\newtheorem{conjecture}[theorem]{Conjecture}
\theoremstyle{remark}
\newtheorem{remark}[theorem]{Remark}

\title[A Conjecture on a P\'olya Functional]{A Proof of a Conjecture on a P\'olya Functional}
\author{Zikang Deng}
\address{Beijing Normal University, Beijing, China}
\date{August 30, 2026}
\hypersetup{
  pdftitle={A Proof of a Conjecture on a P\'olya Functional},
  pdfauthor={Zikang Deng}
}

\begin{document}

\begin{abstract}
Let $\lambda _1(\Omega)$ and $T(\Omega)$ denote the first Dirichlet
eigenvalue and torsional rigidity of a bounded convex domain
$\Omega\subset\R^2$, and let $M=\max_\Omega u$, where $u$ is the torsion
function.  We prove the sharp inequalities
\[
   \frac{\pi^2}{24}<\frac{\lambda _1(\Omega)T(\Omega)}{|\Omega|}
   <\frac{\pi^2}{12},
\]
thereby resolving, in dimension two, Conjecture~4.2 of van den Berg,
Buttazzo, and Pratelli \cite{vandenberg2021}.  Its planar formulation was
restated as Conjecture~1.1 by Ba\~nuelos and Mariano, who proved it for
triangles and rectangles \cite{banuelos2024}.  The lower bound follows by
combining Payne's strict estimate for $\lambda _1M$ with the sharp
torsion-efficiency inequality
\[
       T(\Omega)\geq \frac13 |\Omega|M.
\]
For smooth, strictly convex $\Omega$,
let $\phi$ be an Airy stress potential and put $X=-\nabla\phi$.  Then
$X|_{\partial\Omega}$ parametrizes a convex body $K$ whose surface-area
measure is the torsional first-variation measure.  If $x_0$ maximizes $u$
and $X_0=X(x_0)$, a one-dimensional inequality for every directional
maximum profile gives the stronger containment
\[
       K-X_0\supset \frac{2M}{3}(\Omega-x_0).
\]
For the upper bound, we factor the P\'olya functional through the second
torsion moment.  A level-set torsion--perimeter inequality yields the
factor $5/6$, while a sharp weighted one-dimensional estimate yields the
factor $\pi^2/10$.  Collapsing triangles and elongating rectangles show
that both endpoint constants, as well as the efficiency constant $1/3$,
are sharp.
\end{abstract}

\maketitle

\section{Introduction and main results}

Let $\Omega\subset\R^2$ be a bounded convex domain.  Its torsion function
$u=u_\Omega$ is the weak solution of
\begin{equation}\label{eq:torsion}
     -\Delta u=1\quad\hbox{in }\Omega,
     \qquad u\in H_0^1(\Omega).
\end{equation}
Thus $u$ has zero Sobolev trace on $\partial\Omega$.  The first Dirichlet
eigenvalue is characterized by
\begin{equation}\label{eq:rayleigh-definition}
 \lambda _1(\Omega)=
 \inf_{0\neq v\in H_0^1(\Omega)}
 \frac{\int_\Omega|\nabla v|^2}{\int_\Omega v^2}.
\end{equation}
We use the notation
\[
 A=|\Omega|,\qquad M=\max_\Omega u,
 \qquad T=T(\Omega)=\int_\Omega u.
\]
The scale-invariant P\'olya functional is
\[
             F(\Omega)=\frac{\lambda _1(\Omega)T(\Omega)}{|\Omega|}.
\]
Here $T(\Omega)$ measures the resistance to torsion of a beam with
cross-section $\Omega$, while $\lambda _1(\Omega)$ is the square of its
fundamental membrane frequency.  P\'olya and Szeg\H{o} proved the universal
estimate $F(\Omega)\leq1$ \cite{polyaszego1951}.  The following sharp
problem is the planar case of Conjecture~4.2 in
\cite{vandenberg2021}; it is stated in this form as Conjecture~1.1 in
\cite{banuelos2024}.

\begin{conjecture}[P\'olya functional conjecture]\label{conj:polya}
For every bounded convex planar domain $\Omega\subset\R^2$,
\begin{equation}\label{eq:conjecture}
        \frac{\pi^2}{24}<F(\Omega)<\frac{\pi^2}{12}.
\end{equation}
Both constants are sharp.  The lower constant is approached by collapsing
triangles, and the upper constant by elongating rectangles approaching the
infinite strip.
\end{conjecture}

Before the present argument, Conjecture~\ref{conj:polya} was known for
several special classes.  Van den Berg, Ferone, Nitsch, and Trombetti
proved the lower bound for rhombi and isosceles triangles and determined
the limits for general thinning convex sets \cite{vandenberg2020}.
Ba\~nuelos and Mariano proved both bounds for every triangle and rectangle,
as well as the upper bound for tangential quadrilaterals
\cite{banuelos2024}.  For arbitrary convex planar domains, the best
previously known lower bound was $\pi^2/32$, obtained by combining the
torsion-efficiency estimate of \cite{dellapietra2018} with Payne's
spectral--maximum estimate \cite{payne1981}; the best upper bound was the
value $0.996613\ldots$ obtained by Ftouhi \cite{ftouhi2022}.

Our main result is the following.

\begin{theorem}\label{thm:main}
Every bounded convex domain $\Omega\subset\R^2$ satisfies
\[
       \frac{\pi^2}{24}
       <\frac{\lambda _1(\Omega)T(\Omega)}{|\Omega|}
       <\frac{\pi^2}{12}.
\]
Both constants are sharp.
More precisely, the lower value is approached by the collapsing isosceles
triangles
\[
 \Omega_\varepsilon
 =\{(x,y):0<x<1,\ 0<y<\varepsilon(1-2|x-\tfrac12|)\},
 \qquad \varepsilon\downarrow0,
\]
and the upper value is approached by the rectangles
$R_L=(-L/2,L/2)\times(0,1)$ as $L\to\infty$.
\end{theorem}

The lower bound follows from a sharp resolution of the torsion-efficiency
conjecture of Henrot, Lucardesi, and Philippin \cite{henrot2018}; see also
\cite{dellapietra2018}.

\begin{theorem}[Sharp torsion efficiency]\label{thm:efficiency}
For every bounded convex planar domain,
\begin{equation}\label{eq:efficiency}
             T(\Omega)\geq \frac13|\Omega|M(\Omega).
\end{equation}
The constant $1/3$ is sharp.  More precisely, if $\Omega$ is smooth and
strictly convex, there exist a convex body $K$ and points $x_0\in\Omega$
and $X_0\in K$, where $u(x_0)=M$, such that
\begin{equation}\label{eq:containment-intro}
             K-X_0\supset \frac{2M}{3}(\Omega-x_0).
\end{equation}
\end{theorem}

\subsection{Discussion of the method of proof}

The main difficulty is that the two factors in $F$ respond in opposite
directions to classical symmetrizations: symmetrization decreases the
first Dirichlet eigenvalue but increases torsional rigidity.  Thus a sharp
estimate for their product does not follow by optimizing the two factors
separately.  Our proof treats the lower and upper bounds by two independent
mechanisms.

For the lower bound, write
\[
 F(\Omega)=\bigl(\lambda _1(\Omega)M(\Omega)\bigr)
 \frac{T(\Omega)}{|\Omega|M(\Omega)}.
\]
Payne's strict estimate gives
$\lambda _1(\Omega)M(\Omega)>\pi^2/8$.  Theorem~\ref{thm:efficiency}
supplies the sharp second factor $1/3$.  To prove that theorem for a smooth
strictly convex domain, we form the divergence-free stress tensor
\eqref{eq:stress}, choose an Airy potential $\phi$, and put
$X=-\nabla\phi$.  The boundary identity $dX/ds=g^2t$ shows directly that
$X(\partial\Omega)$ bounds a convex body $K$ whose surface-area measure is
the torsional first-variation measure.  In each direction we then consider
the maximum of $u$ on the perpendicular fibers.  A sharp one-dimensional
inequality for this directional maximum profile yields the support-function
containment \eqref{eq:containment-intro}.  Mixed-area monotonicity converts
that containment into $T(\Omega)\geq |\Omega|M(\Omega)/3$.  General convex
domains follow by smooth outer approximation.

For the upper bound, we use the exact factorization
\[
 F(\Omega)=
 \frac{T(\Omega)^2}{|\Omega|\int_\Omega u^2}
 \frac{\lambda _1(\Omega)\int_\Omega u^2}{T(\Omega)}.
\]
The first factor is strictly smaller than $5/6$.  This follows from the
torsion--perimeter inequality applied to all superlevel sets of $u$.  For
the second factor, a monotone coordinate built from the level distribution
reduces the Rayleigh quotient to a weighted one-dimensional problem; an
explicit sine test function gives the sharp bound $\pi^2/10$.  Their
product is $\pi^2/12$.  Finally, long rectangles and collapsing triangles
give the two limiting constants and also show that the efficiency constant
$1/3$ is sharp.

The measure defining $K$ is the torsional first-variation measure used in
the torsional Minkowski problem \cite{colesanti2010}; related variational
measures are studied in \cite{crasta2024}.  Thus the existence of an
associated Minkowski body is not itself new.  The new ingredients are the
interior Airy-gradient realization of its boundary, the identity for the
directional maximum profiles in Lemma~\ref{lem:airy-ridge}, and the
single-center containment \eqref{eq:containment-intro}.  The Airy potential
itself need not be convex; the boundary curve $X|_{\partial\Omega}$ is the
curve that bounds $K$.

\subsection{Organization of the paper}

Section~\ref{sec:classical} records the concavity, $P$-function,
torsion--perimeter, and mixed-area inputs, including a proof of Payne's
strict estimate.  Section~\ref{sec:ridge} establishes the sharp
one-dimensional inequality for directional maximum profiles.
Section~\ref{sec:airy} constructs $K$ through the Airy realization and
proves Theorem~\ref{thm:efficiency}.  Section~\ref{sec:upper} proves the
upper bound independently through the second torsion moment.
Section~\ref{sec:sharpness} verifies sharpness of all three constants.

\section{Classical inputs and notation}\label{sec:classical}

The Makar--Limanov--Kennington power-concavity theorem gives
\begin{equation}\label{eq:power-concavity}
                         \sqrt{u}\ \hbox{is concave on }\Omega;
\end{equation}
see \cite{makar1971,kennington1985}.  Consequently every superlevel set of
$u$ is convex.

\begin{lemma}\label{lem:unique-max}
The torsion function of a bounded convex planar domain has a unique
maximum point $x_0$.  Moreover,
\[
                 \nabla u(x)\neq0\qquad(x\in\Omega\setminus\{x_0\}).
\]
Thus every $t\in(0,1)$ is a regular value of $U=u/M$.
\end{lemma}

\begin{proof}
If two distinct points maximize $u$, concavity of $\sqrt u$ forces $u$ to
equal $M$ on the segment joining them.  Restrict $u$ to the maximal open
chord of $\Omega$ containing that segment.  Interior analyticity makes this
one-variable restriction analytic; since it is constant on a nondegenerate
subinterval, it is constant on the whole chord.  Continuity up to the two
endpoints then contradicts the zero boundary condition.

If $\nabla u(x)=0$ at an interior point, then
$\nabla\sqrt u(x)=0$.  A differentiable concave function lies below each
of its tangent planes, so $x$ is a global maximum point of $\sqrt u$.
Uniqueness gives $x=x_0$.
\end{proof}

The Payne--Philippin $P$-function estimate is
\begin{equation}\label{eq:p-function}
                   |\nabla u|^2+2u\leq2M\qquad\hbox{in }\Omega.
\end{equation}
For smooth convex domains this follows from the maximum principle; the
estimate for arbitrary convex domains follows by smooth convex outer
approximation, or directly from the nonsmooth-domain maximum principle in
\cite{philippin2010}.  Under smooth outer approximation, the torsion
functions converge locally in $C^1$ and their maxima converge, which
passes \eqref{eq:p-function} to every interior point.

We also use the sharp planar P\'olya torsion--perimeter inequality
\begin{equation}\label{eq:torsion-perimeter}
          \frac{T(D)\Per(D)^2}{|D|^3}>\frac13
\end{equation}
for every bounded convex planar domain $D$.  The constant is approached by
thin rectangles and is not attained by a bounded domain; see
\cite{polya1960,amato2026}.

For completeness, we include Payne's spectral--maximum estimate
\cite{payne1981} with its strictness mechanism.

\begin{proposition}[Payne]\label{prop:payne}
Every bounded convex planar domain satisfies
\begin{equation}\label{eq:payne}
                         \lambda _1(\Omega)M(\Omega)>\frac{\pi^2}{8}.
\end{equation}
The constant is approached by elongating rectangles.
\end{proposition}

\begin{proof}
Put $a=\pi/2$, $z=(1-u/M)^{1/2}$, and
\[
                 w=p(u)=\cos(az).
\]
The function $p$ is smooth on $[0,M]$, $p(0)=0$, and
\[
 p'=\frac{a\sin(az)}{2Mz},\qquad
 p''=\frac{a\{\sin(az)-az\cos(az)\}}{4M^2z^3}>0,
\]
where the values at $z=0$ are understood by continuity.  Direct
differentiation also gives
\begin{equation}\label{eq:p-identity}
             p'-2(M-u)p''=\frac{a^2}{2M}p.
\end{equation}
Since $w\in H_0^1(\Omega)$, the chain rule and
\eqref{eq:p-function} yield, in the weak sense,
\begin{equation}\label{eq:payne-super}
 -\Delta w-\frac{\pi^2}{8M}w
 =p''(u)\bigl(2(M-u)-|\nabla u|^2\bigr)\geq0.
\end{equation}
Let $\varphi_1>0$ be a first Dirichlet eigenfunction.  Pairing
\eqref{eq:payne-super} with $\varphi_1$ gives
\[
 0\leq\left(\lambda _1-\frac{\pi^2}{8M}\right)
             \int_\Omega w\varphi_1,
\]
and hence the non-strict form of \eqref{eq:payne}.

Suppose equality holds.  The nonnegative right-hand side of
\eqref{eq:payne-super} has zero pairing with the strictly positive
$\varphi_1$.  Since $p''>0$, continuity in the interior gives
\[
                    |\nabla u|^2=2(M-u)\quad\hbox{in }\Omega.
\]
On $\Omega\setminus\{x_0\}$ set $r=\sqrt{2(M-u)}$.  Then
$|\nabla r|=1$ and, because $u=M-r^2/2$ and $\Delta u=-1$,
\[
                  |\nabla r|^2+r\Delta r=1,
                  \qquad\Delta r=0.
\]
Bochner's identity gives $D^2r=0$.  The punctured convex domain
$\Omega\setminus\{x_0\}$ is connected in the plane, so $r$ is affine
there.  An affine function vanishing at $x_0$ and nonnegative in a full
neighborhood of $x_0$ has zero gradient, contradicting $|\nabla r|=1$.
Thus equality is impossible.  All steps above are interior or weak
arguments, so no boundary smoothness beyond convexity is needed.
\end{proof}

We use the planar mixed-area convention
\begin{equation}\label{eq:mixed-expansion}
         |C+tD|=|C|+2tV(C,D)+t^2|D|.
\end{equation}
If $S_C$ denotes surface-area measure and $h_C$ the support function, then
\begin{equation}\label{eq:mixed-integral}
  V(C,D)=\frac12\int_{\Sph^1}h_C\dd S_D
        =\frac12\int_{\Sph^1}h_D\dd S_C.
\end{equation}
Mixed area is symmetric, translation invariant, homogeneous, and monotone
in each argument; see \cite{schneider2014}.

\section{A sharp inequality for directional maximum profiles}\label{sec:ridge}

\begin{lemma}[One-sided profile inequality]\label{lem:ridge}
Let $v:[0,L]\to[0,M]$ be absolutely continuous, with $v(0)=0$ and
$v(L)=M$.  Suppose that $\sqrt v$ is concave and nondecreasing and that
\begin{equation}\label{eq:ridge-obstacle-x}
                   v'(x)^2\leq2(M-v(x))\quad\hbox{for a.e. }x.
\end{equation}
Then
\begin{equation}\label{eq:ridge-ineq}
              \int_0^L\bigl(2v-v'^2\bigr)\dd x\geq\frac{2M}{3}L.
\end{equation}
\end{lemma}

\begin{proof}
Set
\[
       q=\sqrt{v/M},\qquad \xi=x/\sqrt M.
\]
If $v=M$ on a terminal interval of length $\ell$, deleting that interval
decreases the left side of \eqref{eq:ridge-ineq} by $2M\ell$ and the right
side by only $2M\ell/3$.  Thus the deficit decreases by
$4M\ell/3$, and it is enough to treat the case without a terminal
plateau.

Now $q$ is strictly increasing from $0$ to $1$.  Its inverse
$\xi=\xi(q)$ is a finite convex function.  Hence it is absolutely
continuous, its a.e. derivative
\[
                        \mathcal A(q)=\frac{d\xi}{dq}
\]
has a nondecreasing representative, and the a.e. inverse derivative
formula is valid.  Since
\[
                v'=\frac{2\sqrt M\,q}{\mathcal A(q)},
\]
condition \eqref{eq:ridge-obstacle-x} is equivalent to
\begin{equation}\label{eq:ridge-obstacle-q}
       \mathcal A(q)\geq\mathcal A_0(q)
       :=\frac{\sqrt2\,q}{\sqrt{1-q^2}}.
\end{equation}
The monotone change-of-variables formula gives
\begin{align}
 \frac1{M^{3/2}}\left\{
   \int_0^L(2v-v'^2)\dd x-\frac{2M}{3}L\right\}
  =\int_0^1\left[
       \left(2q^2-\frac23\right)\mathcal A(q)
       -\frac{4q^2}{\mathcal A(q)}\right]\dd q.
                                                        \label{eq:deficit}
\end{align}

For $s>0$ define the lengths
\[
 Q_-(s)=\bigl|\{q\in(0,1):\mathcal A(q)<s\}\bigr|,
 \qquad
 Q_+(s)=\bigl|\{q\in(0,1):\mathcal A(q)\leq s\}\bigr|.
\]
Because $\mathcal A$ is nondecreasing, these are the lengths of initial
intervals.  They agree except when the level set $\{\mathcal A=s\}$ has
positive measure, and there are at most countably many such values.  Write
their common a.e. value as $Q(s)$.  Since
$\int_0^1(2q^2-2/3)\dd q=0$, Fubini's theorem, applied to an absolutely
integrable function, gives
\begin{equation}\label{eq:layer-one}
 \int_0^1\left(2q^2-\frac23\right)\mathcal A(q)\dd q
   =\frac23\int_0^\infty Q(s)(1-Q(s)^2)\dd s.
\end{equation}
Tonelli's theorem and
$\mathcal A^{-1}=\int_{\mathcal A}^{\infty}s^{-2}\dd s$ give
\begin{equation}\label{eq:layer-two}
 \int_0^1\frac{4q^2}{\mathcal A(q)}\dd q
   =\frac43\int_0^\infty\frac{Q(s)^3}{s^2}\dd s.
\end{equation}
The integrals are legitimate because $\mathcal A\in L^1(0,1)$ and
\eqref{eq:ridge-obstacle-q} controls the second integrand near both
endpoints.  Moreover, the same obstacle implies
\begin{equation}\label{eq:Q-obstacle}
             Q(s)\leq\mathcal A_0^{-1}(s)
                   =\frac{s}{\sqrt{s^2+2}}.
\end{equation}
Combining \eqref{eq:deficit}--\eqref{eq:Q-obstacle} yields
\[
 \frac23\int_0^\infty
 Q(s)\left(1-Q(s)^2-\frac{2Q(s)^2}{s^2}\right)\dd s\geq0,
\]
which proves \eqref{eq:ridge-ineq}.
\end{proof}

\begin{remark}\label{rem:equality-ridge}
With no positive-length terminal plateau, the equality profiles in
Lemma~\ref{lem:ridge} are, after choosing the nondecreasing representative,
\[
       \mathcal A(q)=\max\left\{c,
                   \frac{\sqrt2q}{\sqrt{1-q^2}}\right\}
       \quad\hbox{for a.e. }q,\qquad c\geq0.
\]
Thus the constant is already encoded by a linear $\sqrt v$ branch followed
by a branch saturating the $P$-function obstacle.
Indeed, equality in the final nonnegative integral in the proof forces
\[
 Q(s)\in\left\{0,\frac{s}{\sqrt{s^2+2}}\right\}
 \quad\hbox{for a.e. }s>0.
\]
The monotonicity of $Q$ gives a single threshold $c$ between these two
alternatives, and inversion yields the displayed formula for
$\mathcal A$.  Conversely, substituting that formula into
\eqref{eq:deficit} gives equality.
\end{remark}

\section{Airy realization and torsion efficiency}\label{sec:airy}

We first assume that $\Omega$ is $C^\infty$ and strictly convex.  Let
$n=n_\Omega$ denote its outward unit normal, set
$g=-\partial_n u>0$ on $\partial\Omega$, and define the symmetric stress
\begin{equation}\label{eq:stress}
       \Sigma=(|\nabla u|^2-2u)\id-2\nabla u\otimes\nabla u.
\end{equation}
A direct differentiation using $\Delta u=-1$ gives, with repeated indices
summed,
\[
 \partial_j\Sigma_{ij}
 =2u_k u_{ki}-2u_i-2u_{ij}u_j-2u_i\Delta u=0.
\]
Thus
\begin{equation}\label{eq:div-stress}
                        \operatorname{div}\Sigma=0.
\end{equation}
Let $J$ be counterclockwise rotation through $\pi/2$.

\begin{lemma}[Airy realization]\label{lem:airy-body}
There is an Airy potential $\phi\in C^\infty(\overline\Omega)$, unique
modulo affine functions,
such that
\begin{equation}\label{eq:airy}
                         \Sigma=J^TD^2\phi J.
\end{equation}
If $X=-\nabla\phi$, then $X|_{\partial\Omega}$ parametrizes, up to
translation, the boundary of a smooth strictly convex body $K$, and
\begin{equation}\label{eq:surface-measure}
       S_K=(n_\Omega)_\#(g^2\dd s).
\end{equation}
Here $\#$ denotes push-forward of measures.
\end{lemma}

\begin{proof}
Write $\Sigma=\left(\begin{smallmatrix}a&b\\b&c\end{smallmatrix}\right)$.
The two divergence equations say
$a_x+b_y=0$ and $b_x+c_y=0$.  Hence the one-forms
$c\dd x-b\dd y$ and $-b\dd x+a\dd y$ are closed.  Since a convex domain
is simply connected, they are $\dd p$ and $\dd q$ for some functions
$p,q$.  The identity $p_y=q_x=-b$ then gives a function $\phi$ with
$\phi_x=p$, $\phi_y=q$.  Thus
$\phi_{xx}=c$, $\phi_{xy}=-b$, $\phi_{yy}=a$, which is
\eqref{eq:airy}.  The ambiguity is affine.
Because $u\in C^\infty(\overline\Omega)$ for a smooth domain, the
coefficients of $\Sigma$ are smooth up to the boundary, and the preceding
primitive construction gives $\phi\in C^\infty(\overline\Omega)$.

On $\partial\Omega$, $\nabla u=-gn$ and therefore
\begin{equation}\label{eq:traction}
                         \Sigma n=-g^2n.
\end{equation}
If $t=Jn$ is the counterclockwise unit tangent and $s$ is arclength, then
\eqref{eq:airy}--\eqref{eq:traction} give
\begin{equation}\label{eq:X-boundary}
       \frac{dX}{ds}=-D^2\phi\,t=g^2t.
\end{equation}
Because $X=-\nabla\phi$ is single-valued, its restriction to the closed
curve $\partial\Omega$ is closed.  Equation \eqref{eq:X-boundary} and
$g>0$ show that it is a regular $C^\infty$ curve with the same tangent and
outward normal as $\partial\Omega$.  Its tangent angle is nondecreasing,
has total increment $2\pi$, and is constant on no nontrivial interval,
because $\Omega$ is strictly convex.  The tangent-angle characterization
of convex curves therefore shows that $X(\partial\Omega)$ is the
positively oriented boundary of a smooth strictly convex body $K$.
Moreover, its arclength element is $\dd s_K=g^2\dd s$, so directly
\[
                  S_K=(n_\Omega)_\#(g^2\dd s).
\]
In particular, the boundary points of $\Omega$ with normal $\pm e$ are
carried to the support points of $K$ with the same respective normals.
The same conclusion also follows from the balance identity: the divergence
theorem and \eqref{eq:traction} give
\[
       \int_{\Sph^1}\theta\dd S_K(\theta)
       =\int_{\partial\Omega}g^2n\dd s
       =-\int_{\partial\Omega}\Sigma n\dd s=0.
\]
The planar Minkowski uniqueness theorem also shows that this $K$ is
determined by \eqref{eq:surface-measure} up to translation.
\end{proof}

We next compare the support functions of $K$ and $\Omega$.  Fix
$e\in\Sph^1$ and use positively oriented coordinates $(e,Je)$, with
coordinates $(x,y)$.  Let $(a,b)$ be the projection of $\Omega$ on the
$x$-axis and define
\begin{equation}\label{eq:fibre-max}
                 v(x)=\max\{u(x,y):(x,y)\in\Omega\},
                 \qquad a<x<b.
\end{equation}
We call $v$ the directional maximum profile associated with $e$.
The hypograph of $\sqrt v$ is the projection of the convex hypograph of
$\sqrt u$, so
\begin{equation}\label{eq:v-concave}
                         \sqrt v\ \hbox{is concave on }(a,b).
\end{equation}
The fiber maximizer $y=y(x)$ is unique.  Indeed, two maximizers would make
$\sqrt u(x,\cdot)$ constant between them; one-variable analyticity would
then make $u$ constant on the whole open fiber, contrary to the boundary
values.  Strict convexity makes the support points over $a$ and $b$ unique.
Compactness and $u\in C(\overline\Omega)$ show that as $x\to a$ or
$x\to b$, every fiberwise maximizing point tends to the corresponding
support point.
Consequently $v$ extends continuously to $[a,b]$ with
\begin{equation}\label{eq:v-endpoints}
                         v(a)=v(b)=0.
\end{equation}

Define
\begin{equation}\label{eq:z-def}
                         z(x)=X(x,y(x))\cdot e.
\end{equation}

\begin{lemma}[Identity for directional maximum profiles]\label{lem:airy-ridge}
The function $z$ is locally Lipschitz on $(a,b)$ and
\begin{equation}\label{eq:ridge-identity}
                         z'(x)=2v(x)-v'(x)^2
                         \quad\hbox{for a.e. }x\in(a,b).
\end{equation}
Moreover, $z$ is $2M$-Lipschitz on $(a,b)$ and extends continuously to
$[a,b]$.
\end{lemma}

\begin{proof}
On a compact subinterval of $(a,b)$ all fiberwise maximizing points remain
in a compact subset of $\Omega$.  There, bound $|u_x|$ and
$|\phi_{xx}|$ by a constant $C$.  From \eqref{eq:stress} and
\eqref{eq:airy},
\[
                 \partial_y(X\cdot e)=-\phi_{xy}=-2u_xu_y.
\]
Along each fiber, concavity of $\sqrt u$ makes $u$ unimodal.  Hence
\begin{equation}\label{eq:vertical-X}
 |X(x,y)\cdot e-z(x)|
 \leq2C\,|v(x)-u(x,y)|.
\end{equation}
For $x_2=x_1+h$ and $|h|$ sufficiently small, the point
$(x_2,y(x_1))$ lies in the neighboring fiber, and
\[
 0\leq v(x_2)-u(x_2,y(x_1))
 \leq |v(x_2)-v(x_1)|
       +|u(x_2,y(x_1))-u(x_1,y(x_1))|=O(|h|).
\]
Together with \eqref{eq:vertical-X} and the horizontal bound on
$X\cdot e$, this proves local Lipschitz continuity.

At almost every $x$, both $v$ and $z$ are differentiable.  Danskin's
formula and fiberwise maximality give
\begin{equation}\label{eq:danskin}
       v'(x)=u_x(x,y(x)),\qquad u_y(x,y(x))=0.
\end{equation}
At such an $x$, the vertical error in the difference quotient is $o(h)$:
indeed,
$v(x+h)-u(x+h,y(x))=o(h)$ by \eqref{eq:danskin}, and then
\eqref{eq:vertical-X} applies.  Therefore
\[
 z'(x)=-\phi_{xx}(x,y(x)).
\]
Since $\Sigma_{yy}=\phi_{xx}=u_x^2-u_y^2-2u$, formula
\eqref{eq:ridge-identity} follows from \eqref{eq:danskin}.

Finally, \eqref{eq:p-function} gives $v'^2\leq2(M-v)$, and hence
\[
       -2M\leq2v-v'^2\leq2M.
\]
On every compact subinterval, integration of
\eqref{eq:ridge-identity} gives the same $2M$ Lipschitz bound.  It follows
for every pair of interior points, and therefore gives the asserted
extension to the endpoints.
\end{proof}

\begin{proposition}[Support-function containment]\label{prop:containment}
Let $\Omega$ be smooth and strictly convex, let $K$ be the body supplied
by Lemma~\ref{lem:airy-body}, and let $x_0$ be the maximum point of $u$.
With $X_0=X(x_0)$, one has
\[
                   K-X_0\supset\frac{2M}{3}(\Omega-x_0).
\]
\end{proposition}

\begin{proof}
Let $x^*=x_0\cdot e$ and put $X_0=X(x_0)$.  At $x=x^*$ the fiberwise
maximizing point is $x_0$.  By Lemma~\ref{lem:airy-body}, continuity of
$X$ to the boundary, and the endpoint convergence of the fiberwise
maximizing points,
\begin{align}
 X_0\cdot e-\min_{k\in K}k\cdot e
   &=\int_a^{x^*}(2v-v'^2)\dd x,                  \label{eq:left-ridge}\\
 \max_{k\in K}k\cdot e-X_0\cdot e
   &=\int_{x^*}^{b}(2v-v'^2)\dd x.               \label{eq:right-ridge}
\end{align}
The function $\sqrt v$ is nondecreasing on $[a,x^*]$ and nonincreasing on
$[x^*,b]$, and \eqref{eq:p-function} gives
\begin{equation}\label{eq:v-obstacle}
                         v'^2\leq2(M-v).
\end{equation}
Concavity of $\sqrt v$ and the endpoint continuity in
\eqref{eq:v-endpoints} make $\sqrt v$, and hence $v$, absolutely
continuous on each of these two flanks.
Apply Lemma~\ref{lem:ridge} to the left flank and to the reversed right
flank.  Equations \eqref{eq:left-ridge}--\eqref{eq:right-ridge} give
\[
 X_0\cdot e-\min_{k\in K}k\cdot e
     \geq\frac{2M}{3}(x^*-a),
 \qquad
 \max_{k\in K}k\cdot e-X_0\cdot e
     \geq\frac{2M}{3}(b-x^*).
\]
In support-function form, for every $e\in\Sph^1$,
\begin{equation}\label{eq:support-containment}
              h_{K-X_0}(e)\geq\frac{2M}{3}h_{\Omega-x_0}(e).
\end{equation}
Thus
\begin{equation}\label{eq:containment}
                   K-X_0\supset\frac{2M}{3}(\Omega-x_0).
\end{equation}
In particular, the right side contains the origin, so
\eqref{eq:containment} also proves $X_0\in K$.
\end{proof}

\begin{proof}[Proof of Theorem~\ref{thm:efficiency}]
We first treat a smooth strictly convex domain.  Multiply
$-\Delta u=1$ by $x\cdot\nabla u$ and integrate.  Since $u=0$ on the
boundary, the right-hand side is
$\int_\Omega x\cdot\nabla u=-2T$.  Integrating the left-hand side by
parts and using $\nabla u=-gn$ on $\partial\Omega$ gives, specifically in
dimension two,
\[
 \int_\Omega(-\Delta u)(x\cdot\nabla u)
 =-\frac12\int_{\partial\Omega}(x\cdot n)g^2\dd s.
\]
Consequently the planar Pohozaev identity is
\begin{equation}\label{eq:pohozaev}
             \int_{\partial\Omega}(x\cdot n)g^2\dd s=4T.
\end{equation}
By \eqref{eq:mixed-integral} and \eqref{eq:surface-measure},
\begin{equation}\label{eq:mixed-torsion}
 V(\Omega,K)=\frac12\int_{\partial\Omega}(x\cdot n)g^2\dd s=2T.
\end{equation}
Translation invariance, monotonicity, and \eqref{eq:containment} yield
\[
 2T=V(\Omega-x_0,K-X_0)
   \geq\frac{2M}{3}V(\Omega-x_0,\Omega-x_0)
   =\frac{2M}{3}A.
\]
This proves \eqref{eq:efficiency} for smooth strictly convex domains.

For a general bounded convex domain, place the origin in its interior and
choose smooth strictly convex bodies $\Omega_j$ with
\[
             \Omega\subset\Omega_j\subset(1+\varepsilon_j)\Omega,
             \qquad\varepsilon_j\downarrow0.
\]
Let $A_j$, $M_j$, and $T_j$ denote the corresponding area, maximum, and
torsional rigidity.  Domain monotonicity and the scaling laws give
\[
 A\leq A_j\leq(1+\varepsilon_j)^2A,
 \quad M\leq M_j\leq(1+\varepsilon_j)^2M,
 \quad T\leq T_j\leq(1+\varepsilon_j)^4T.
\]
Hence $A_j\to A$, $M_j\to M$, and $T_j\to T$.  Passing to the limit in
$T_j\geq A_jM_j/3$ proves \eqref{eq:efficiency} for every bounded convex
planar domain.  The collapsing triangles in Section~\ref{sec:sharpness}
show that $1/3$ is sharp.
\end{proof}

The containment in Proposition~\ref{prop:containment} has further
consequences.

\begin{corollary}\label{cor:airy-consequences}
In the smooth strictly convex setting, for every planar convex body $Q$,
\begin{equation}\label{eq:mixed-cor}
                    V(Q,K)\geq\frac{2M}{3}V(Q,\Omega).
\end{equation}
In particular,
\begin{equation}\label{eq:per-area-cor}
  \Per(K)\geq\frac{2M}{3}\Per(\Omega),
  \qquad |K|\geq\frac{4M^2}{9}|\Omega|.
\end{equation}
Moreover,
\begin{equation}\label{eq:K-area}
                   |K|=\int_\Omega(4u^2-|\nabla u|^4),
\end{equation}
and hence
\begin{equation}\label{eq:integral-cor}
                   \int_\Omega(4u^2-|\nabla u|^4)
                   \geq\frac{4M^2}{9}|\Omega|.
\end{equation}
Finally, there are a compact convex set $L$ and a point $Y_0$ such that
\begin{equation}\label{eq:complement-body}
 (K-X_0)+(L-Y_0)=2M(\Omega-x_0),
 \qquad L-Y_0\subset\frac{4M}{3}(\Omega-x_0).
\end{equation}
\end{corollary}

\begin{proof}
Formula \eqref{eq:mixed-cor} is mixed-area monotonicity applied to
\eqref{eq:containment}; perimeter monotonicity and area monotonicity give
\eqref{eq:per-area-cor}.  With $Y\wedge Z=Y_1Z_2-Y_2Z_1$, Green's formula
and \eqref{eq:X-boundary} give
\[
 |K|=\frac12\int_{\partial\Omega}X\wedge\dd X
     =\int_\Omega\det(DX)=\int_\Omega\det\Sigma.
\]
The matrix determinant lemma applied to \eqref{eq:stress} yields
$\det\Sigma=4u^2-|\nabla u|^4$, proving \eqref{eq:K-area} and
\eqref{eq:integral-cor}.

On the boundary, \eqref{eq:p-function} says $g^2\leq2M$.  The measure
$(n_\Omega)_\#((2M-g^2)\dd s)$ is nonnegative and balanced, hence it is
the surface-area measure of a compact convex set $L$.  In the plane,
surface-area measures add under Minkowski addition, so translations may
be chosen such that $K+L=2M\Omega$.  Set $Y_0=2Mx_0-X_0$ and subtract
\eqref{eq:support-containment} from the support-function identity for this
sum to obtain \eqref{eq:complement-body}.
\end{proof}

\begin{proposition}[Lower bound for the P\'olya functional]
\label{prop:lower-bound}
Every bounded convex planar domain satisfies
\[
                  \frac{\lambda _1T}{A}>\frac{\pi^2}{24}.
\]
\end{proposition}

\begin{proof}
Combine Theorem~\ref{thm:efficiency} with the strict Payne inequality
\eqref{eq:payne}.  This gives
\begin{equation}\label{eq:lower-main}
 \frac{\lambda _1T}{A}
   =(\lambda _1M)\frac{T}{AM}
   >\frac{\pi^2}{8}\cdot\frac13
   =\frac{\pi^2}{24}.
\end{equation}
\end{proof}

\section{The upper bound through the second torsion moment}\label{sec:upper}

Put $U=u/M$ and define the normalized level distribution
\begin{equation}\label{eq:level-defs}
 S(t)=\frac1A|\{U>t\}|,
 \qquad I(t)=\int_t^1S(s)\dd s,
 \qquad \Phi=I(0)=\frac{T}{AM}.
\end{equation}
By Lemma~\ref{lem:unique-max}, every $t\in(0,1)$ is a regular value.
Consequently the compact level hypersurfaces vary smoothly on compact
subintervals of $(0,1)$, and
\begin{equation}\label{eq:S-smooth}
                         S\in C^\infty_{\mathrm{loc}}(0,1).
\end{equation}

\subsection{The first factor}

\begin{lemma}[Second-moment ratio]\label{lem:first-factor}
For every bounded convex planar domain,
\begin{equation}\label{eq:first-factor-statement}
                       \frac{T^2}{A\int_\Omega u^2}<\frac56.
\end{equation}
\end{lemma}

\begin{proof}
Let $\Omega_t=\{U>t\}$.  The function $u-Mt$ is the torsion function of
$\Omega_t$, so
\begin{equation}\label{eq:level-torsion}
        |\Omega_t|=AS(t),\qquad T(\Omega_t)=AM I(t).
\end{equation}
If $P_t=\Per(\Omega_t)$, flux and coarea give
\begin{equation}\label{eq:flux-coarea}
 \int_{\partial\Omega_t}|\nabla u|\dd s=AS(t),
 \qquad -AS'(t)=M\int_{\partial\Omega_t}|\nabla u|^{-1}\dd s.
\end{equation}
Set $\eta=M/A$ and $p_t=P_t/\sqrt A$.  Cauchy--Schwarz and
\eqref{eq:flux-coarea} imply
\begin{equation}\label{eq:per-upper}
                         p_t^2\leq\frac{S(t)(-S'(t))}{\eta}.
\end{equation}
Applying \eqref{eq:torsion-perimeter} to $\Omega_t$ and using
\eqref{eq:level-torsion},
\begin{equation}\label{eq:per-lower}
                         \frac{\eta I(t)p_t^2}{S(t)^3}>\frac13.
\end{equation}
Combining \eqref{eq:per-upper} and \eqref{eq:per-lower},
\begin{equation}\label{eq:I-convex-diff}
      I(t)(-S'(t))>\frac{S(t)^2}{3},
      \qquad\text{equivalently}\qquad
      I(t)I''(t)>\frac{I'(t)^2}{3}.
\end{equation}
Indeed,
\[
 (I^{2/3})''=\frac23 I^{-4/3}
 \left(I(t)I''(t)-\frac13I'(t)^2\right)>0.
\]
Thus $I^{2/3}$ is strictly convex on $(0,1)$ and extends convexly to
$[0,1]$.  Since $S(0+)=1$, $I(0)=\Phi$, $I'(0+)=-1$, and $I(1)=0$, its
tangent at the origin gives
\begin{equation}\label{eq:I-tangent}
      I(t)\geq\Phi\left(1-\frac{2t}{3\Phi}\right)_+^{3/2},
      \qquad \Phi\leq\frac23.
\end{equation}
The inequality is strict for $t>0$ before the positive part vanishes.
Integration yields
\begin{equation}\label{eq:I-integral}
                         \int_0^1I(t)\dd t>\frac{3\Phi^2}{5}.
\end{equation}
On the other hand, layer cake and Fubini give
\[
  \frac1{AM^2}\int_\Omega u^2
   =2\int_0^1tS(t)\dd t=2\int_0^1I(t)\dd t.
\]
Therefore
\begin{equation}\label{eq:first-factor}
                       \frac{T^2}{A\int_\Omega u^2}<\frac56.
\end{equation}
\end{proof}

\subsection{A weighted one-dimensional coordinate}

Since $S(t)>0$ for $0\leq t<1$, the function $I$ is strictly decreasing.
Introduce the decreasing coordinate and its inverse by
\begin{equation}\label{eq:y-coordinate}
       y(t)=\left(\frac{I(t)}{\Phi}\right)^{1/3},
       \qquad t=t(y),
\end{equation}
and, initially for $0<y<1$, define
\begin{equation}\label{eq:a-weight}
                         a(y)=\frac{S(t(y))}{3\Phi y}.
\end{equation}
We use its monotone one-sided representative at the endpoints; in
particular $a(1)=1/(3\Phi)$ and
$ya(y)=S(t(y))/(3\Phi)\to0$ as $y\downarrow0$.
Since $I=\Phi y^3$ and $I'=-S$,
\begin{equation}\label{eq:y-identities}
       \frac{dt}{dy}=-\frac{y}{a(y)},
       \qquad S(t(y))=3\Phi ya(y),
       \qquad \int_0^1\frac{y}{a(y)}\dd y=1.
\end{equation}
Moreover,
\[
        \frac{d}{dt}I(t)^{2/3}=-2\Phi^{2/3}a(y(t)).
\]
Convexity of $I^{2/3}$, together with the reversal of orientation between
$t$ and $y$, shows that $a$ is nondecreasing.  The Stieltjes-measure
change of variables is most precisely expressed as follows: for every
bounded Borel function $\psi$,
\begin{equation}\label{eq:stieltjes-push}
 \int_{[0,1]}\psi(t)(-\dd S(t))
 =3\Phi\int_{[0,1]}\psi(t(y))\dd(ya(y)).
\end{equation}
To see this directly, use $S(t(y))=3\Phi ya(y)$ and the fact that
$t:[0,1]\to[0,1]$ reverses orientation.  Thus the push-forward of the
measure $3\Phi\,\dd(ya)$ under $y\mapsto t(y)$ is precisely the measure
$-\dd S$, which is equivalent to \eqref{eq:stieltjes-push}.

For every Lipschitz function $H:[0,1]\to\R$ with $H(0)=0$, coarea and
flux applied to the torsion-coordinate trial function $H(U)$ give the
Rayleigh bound
\begin{equation}\label{eq:rayleigh-t}
       \lambda _1M\leq
       \frac{\int_0^1S(t)H'(t)^2\dd t}
            {\int_{[0,1]}H(t)^2(-\dd S(t))}.
\end{equation}
Writing $f(y)=H(t(y))$ and using
\eqref{eq:y-identities}--\eqref{eq:stieltjes-push}, we obtain
\begin{equation}\label{eq:rayleigh-y}
       \lambda _1M\leq
       \frac{\int_0^1a(y)^2f'(y)^2\dd y}
            {\int_{[0,1]}f(y)^2\dd(ya(y))}.
\end{equation}
If $m_k=A^{-1}\int_\Omega U^k$, then another use of
\eqref{eq:y-identities} gives $m_1=\Phi$ and
\[
 m_2=2\int_0^1tS(t)\dd t
    =6\Phi\int_0^1t(y)y^2\dd y
    =2\Phi\int_0^1\frac{y^4}{a(y)}\dd y.
\]
Consequently
\begin{equation}\label{eq:kappa}
       \frac{m_2}{m_1}=2\int_0^1\frac{y^4}{a(y)}\dd y=:2\kappa.
\end{equation}

\subsection{The sharp weighted estimate}

Set
\begin{equation}\label{eq:test-h}
                         h(y)=\sin\frac{\pi y}{2}.
\end{equation}
Define
\begin{equation}\label{eq:r-def}
 r(s)=2h(s)\int_0^s yh(y)\dd y-\frac{10}{\pi^2}s^4
\end{equation}
and its primitive
\begin{equation}\label{eq:R-def}
 R(x)=\int_0^xr(s)\dd s
 =\frac2{\pi^3}\left[\pi x(2+\cos(\pi x)-x^4)-3\sin(\pi x)\right].
\end{equation}

\begin{lemma}\label{lem:R-positive}
$R(x)\geq0$ for $0\leq x\leq1$, and $R(0)=R(1)=0$.
\end{lemma}

\begin{proof}
For $0<z\leq\pi$, put
\[
                 q(z)=\frac{z(2+\cos z)-3\sin z}{z^5}.
\]
Then
\[
 z^6q'(z)=(15-z^2)\sin z-z(8+7\cos z)=-G(z),
\]
where, explicitly,
\[
                 G(z)=z(8+7\cos z)-(15-z^2)\sin z.
\]
A direct differentiation gives
\[
 G'''(z)=z(\sin z-z\cos z)\geq0,
 \qquad G(0)=G'(0)=G''(0)=0.
\]
The inequality follows because
$\sin z-z\cos z=\int_0^zs\sin s\dd s\geq0$ on $[0,\pi]$.
Thus $G\geq0$, $q$ is nonincreasing, and
$q(z)\geq q(\pi)=\pi^{-4}$.  After setting $z=\pi x$, this is exactly
the nonnegativity of \eqref{eq:R-def}; the endpoint identities are direct.
\end{proof}

\begin{proposition}[Weighted one-dimensional factor]
\label{prop:second-factor}
For every bounded convex planar domain,
\begin{equation}\label{eq:second-factor-statement}
                       \frac{\lambda _1\int_\Omega u^2}{T}
                       \leq\frac{\pi^2}{10}.
\end{equation}
\end{proposition}

\begin{proof}
Let
\begin{equation}\label{eq:test-f}
                         f(y)=\int_y^1\frac{h(s)}{a(s)}\dd s.
\end{equation}
This gives an admissible trial in \eqref{eq:rayleigh-y}.  Indeed,
$h(y)\leq\pi y/2$ and \eqref{eq:y-identities} show that
\[
               f(0)\leq\frac\pi2\int_0^1\frac{y}{a(y)}\dd y
                      =\frac\pi2.
\]
In the original coordinate,
\[
                H'(t)=\frac{h(y)}{y}\leq\frac\pi2,
                \qquad H(0)=f(1)=0.
\]
Thus $H$ is Lipschitz and $H(U)\in H_0^1(\Omega)$.  The numerator is
\begin{equation}\label{eq:test-numerator}
                \int_0^1a^2f'^2\dd y=\int_0^1h^2\dd y=\frac12.
\end{equation}

For the denominator, note that $ya(y)=S(t(y))/(3\Phi)\to0$ as
$y\downarrow0$.  Lebesgue--Stieltjes integration by parts, followed by
Fubini, gives
\begin{align}
 D:=\int_{[0,1]}f^2\dd(ya)
  &=-2\int_0^1yaf f'\dd y                                      \notag\\
  &=2\int_0^1\frac{h(s)}{a(s)}
       \left(\int_0^s yh(y)\dd y\right)\dd s.                  \label{eq:D}
\end{align}
There is no boundary term: $f(1)=0$, while $f(0)<\infty$ and $ya(y)\to0$.
Since $a$ is nondecreasing, $w=1/a$ is nonincreasing.  Apply
Lebesgue--Stieltjes integration by parts on $[\varepsilon,1]$ and let
$\varepsilon\downarrow0$:
\begin{equation}\label{eq:r-weight}
                    \int_0^1\frac{r(s)}{a(s)}\dd s
                    =-\int_{[0,1]}R(s)\dd w(s)\geq0.
\end{equation}
The endpoint at $1$ vanishes because $R(1)=0$ and
$w(1)=3\Phi<\infty$.  At $0$, monotonicity and
\eqref{eq:y-identities} give
\[
       \frac{y^2w(y)}2\leq\int_0^ysw(s)\dd s\longrightarrow0,
\]
while $R(y)=O(y^5)$; hence $R(y)w(y)\to0$.
Equations \eqref{eq:D}, \eqref{eq:r-def}, and \eqref{eq:r-weight} imply
\begin{equation}\label{eq:D-lower}
                          D\geq\frac{10}{\pi^2}\kappa.
\end{equation}
Using \eqref{eq:test-numerator} and \eqref{eq:D-lower} in
\eqref{eq:rayleigh-y}, and then using \eqref{eq:kappa}, gives
\[
       2\kappa\,\lambda _1M\leq\frac{\kappa}{D}\leq\frac{\pi^2}{10},
\]
or equivalently
\begin{equation}\label{eq:second-factor}
                       \frac{\lambda _1\int_\Omega u^2}{T}
                       \leq\frac{\pi^2}{10}.
\end{equation}
\end{proof}

\begin{proof}[Proof of the inequalities in Theorem~\ref{thm:main}]
The exact factorization
\[
 \frac{\lambda _1T}{A}
 =\left(\frac{T^2}{A\int_\Omega u^2}\right)
  \left(\frac{\lambda _1\int_\Omega u^2}{T}\right)
\]
and \eqref{eq:first-factor}, \eqref{eq:second-factor} yield
\begin{equation}\label{eq:upper-main}
                         \frac{\lambda _1T}{A}<\frac56\cdot
                         \frac{\pi^2}{10}=\frac{\pi^2}{12}.
\end{equation}
Together with \eqref{eq:lower-main}, this proves the two strict
inequalities in Theorem~\ref{thm:main}.
\end{proof}

\section{Sharpness}\label{sec:sharpness}

For the upper constant, consider
\[
                   R_L=(-L/2,L/2)\times(0,1).
\]
Separation of variables gives
\begin{equation}\label{eq:rectangle-eigen}
                   \lambda _1(R_L)=\pi^2\left(1+\frac1{L^2}\right).
\end{equation}
The classical rectangular torsion series \cite[p.~108]{polyaszego1951} is
\begin{equation}\label{eq:rectangle-torsion}
 T(R_L)=\frac{L}{12}-\frac{16}{\pi^5}
   \sum_{k=0}^\infty
   \frac{\tanh((2k+1)\pi L/2)}{(2k+1)^5}.
\end{equation}
The series is uniformly bounded in $L$, so
\[
                 \frac{T(R_L)}{|R_L|}=\frac1{12}+O(L^{-1}).
\]
The same family also shows that the two intermediate constants in the
upper-bound proof are sharp.  Let $w(y)=y(1-y)/2$, the torsion function of
the infinite strip, and let $h_L=w-u_{R_L}$.  Then $h_L$ is harmonic,
vanishes on the horizontal sides, and equals $w$ on the vertical sides.
If
\[
 C=\max_{0\leq y\leq1}\frac{w(y)}{\sin(\pi y)},
\]
with the endpoint values interpreted by limits, the maximum principle
gives
\[
 0\leq h_L(x,y)\leq
 C\frac{\cosh(\pi x)}{\cosh(\pi L/2)}\sin(\pi y).
\]
The integral of the right-hand side over $R_L$ is bounded independently
of $L$.  Since $0\leq u_{R_L}\leq w\leq1/8$, it follows that
\[
 \left|\int_{R_L}u_{R_L}^2
       -L\int_0^1w(y)^2\dd y\right|=O(1).
\]
As $\int_0^1w=1/12$ and $\int_0^1w^2=1/120$, we obtain
\begin{equation}\label{eq:rectangle-factors}
 \frac{T(R_L)^2}{|R_L|\int_{R_L}u_{R_L}^2}\longrightarrow\frac56,
 \qquad
 \frac{\lambda _1(R_L)\int_{R_L}u_{R_L}^2}{T(R_L)}
 \longrightarrow\frac{\pi^2}{10}.
\end{equation}
Combining \eqref{eq:rectangle-eigen} with the torsion asymptotic above
yields
\begin{equation}\label{eq:upper-sharp}
                         F(R_L)\longrightarrow\frac{\pi^2}{12}.
\end{equation}
Comparison with the infinite-strip torsion function $y(1-y)/2$ gives
$M(R_L)\leq1/8$; monotone exhaustion of the strip and interior convergence
give $u_{R_L}(0,1/2)\to1/8$.  Hence $M(R_L)\to1/8$, and the same family
also shows the sharpness of Payne's constant.

For the lower constant and the efficiency constant, it is enough to use a
fixed collapsing triangular family.  Let
\[
 \Omega_\varepsilon=\{(x,y):0<x<1,\ 0<y<\varepsilon h(x)\},
 \qquad h(x)=1-2|x-\tfrac12|.
\]
Standard thin-domain asymptotics
\cite{borisov2010,borisov2013,vandenberg2020,banuelos2024} give the first
two limits below.  For the third, comparison with the strip of height
$\varepsilon$ gives $M(\Omega_\varepsilon)\leq\varepsilon^2/8$.  For each
fixed $\delta\in(0,1)$, the domain contains a rectangle of width $\delta$
and height $\varepsilon(1-\delta)$; the long-rectangle limit proved above
then gives the matching lower bound after letting first
$\varepsilon\downarrow0$ and then $\delta\downarrow0$.  Hence
\begin{equation}\label{eq:thin-asymptotics}
 \varepsilon^2\lambda _1(\Omega_\varepsilon)\longrightarrow\pi^2,
 \qquad
 \varepsilon^{-3}T(\Omega_\varepsilon)\longrightarrow
       \frac1{12}\int_0^1h^3,
 \qquad
 \varepsilon^{-2}M(\Omega_\varepsilon)\longrightarrow\frac18.
\end{equation}
Also $|\Omega_\varepsilon|=\varepsilon\int_0^1h$.  Since
\[
                   \int_0^1h=\frac12,
                   \qquad\int_0^1h^3=\frac14,
\]
equations \eqref{eq:thin-asymptotics} imply
\begin{equation}\label{eq:lower-eff-sharp}
       F(\Omega_\varepsilon)\longrightarrow\frac{\pi^2}{24},
       \qquad
       \frac{T(\Omega_\varepsilon)}
            {|\Omega_\varepsilon|M(\Omega_\varepsilon)}
       \longrightarrow\frac13.
\end{equation}
This proves sharpness in Theorems~\ref{thm:main} and
\ref{thm:efficiency}.

\end{document}